\documentclass[12pt]{article}
\usepackage{amssymb,amsmath,amsthm,mathrsfs}
\usepackage[usenames]{color}
\usepackage{hyperref}
\usepackage[all]{xy}
\usepackage{xypic}
\usepackage{graphicx}
\usepackage{amscd}

\baselineskip 18pt \textwidth 16cm  \theoremstyle{plain}

\newtheorem{theorem}{Theorem}[section]
\newtheorem{proposition}[theorem]{Proposition}
\newtheorem{corollary}[theorem]{Corollary}
\newtheorem{lemma}[theorem]{Lemma}
\newtheorem{definition}[theorem]{Definition}

\newtheorem{example}[theorem]{Example}
\newtheorem{remark}[theorem]{Remark}

\newcommand{\Z}{\mathbb Z}

\newcommand{\RamiAlt}[1]{}
\newcommand{\RamiAbout}[1]{}

\DeclareMathOperator{\GL}{GL}

\DeclareMathOperator{\tr}{tr}
\DeclareMathOperator{\Aut}{Aut}
\DeclareMathOperator{\Hom}{Hom}

\DeclareMathOperator{\Bor}{Bor}
\DeclareMathOperator{\Vect}{Vect}

\DeclareMathOperator{\Irr}{Irr}

\usepackage[usenames,dvipsnames]{xcolor}

\definecolor{dblue}{RGB}{6,69,173}
\definecolor{lblue}{RGB}{11,0,128}

\newcommand{\colorlinks}{true}

\newcommand{\linkcolor}{lblue}
\newcommand{\citecolor}{green}
\newcommand{\urlcolor}{dblue}

\newcommand{\linkbordercolor}{red}
\newcommand{\citebordercolor}{green}
\newcommand{\urlbordercolor}{cyan}

\usepackage{hyperref}
\hypersetup{colorlinks=\colorlinks,linkbordercolor=\linkbordercolor, linkcolor=\linkcolor, citecolor=\citecolor, urlcolor=\urlcolor,urlbordercolor=\urlbordercolor,citebordercolor=\citebordercolor}%

\newcommand{\hrefHid}[2]{
\hypersetup{urlbordercolor={1 1 1}}%
\hypersetup{urlcolor=black}%
\href{#1}{#2}%
\hypersetup{urlbordercolor=\urlbordercolor}%
\hypersetup{urlcolor=\urlcolor}%
}

\newcommand{\inhref}[2]{\hyperref[#1]{#2}}

\newcommand{\inhrefHid}[2]{%
\hypersetup{linkbordercolor={1 1 1}}%
\hypersetup{linkcolor=black}%
\inhref{#1}{#2}%
\hypersetup{linkbordercolor=\linkbordercolor}%
\hypersetup{linkcolor=\linkcolor}%
}

\newcommand{\defHref}[3]{\newcommand{#1}[1][#3]{\href{#2}{##1}}}
\newcommand{\defInhref}[3]{\newcommand{#1}[1][#3]{\inhref{#2}{##1}}}
\newcommand{\defHrefHid}[3]{\newcommand{#1}[1][#3]{\hrefHid{#2}{##1}}}
\newcommand{\defInhrefHid}[3]{\newcommand{#1}[1][#3]{\inhrefHid{#2}{##1}}}

\newcommand{\defHrefBoth}[3]{%
\expandafter\defHrefHid \csname #3Hid\endcsname {#1}{#2}%
\expandafter\defHref \csname #3Vis\endcsname {#1}{#2}%
}
\newcommand{\defInhrefBoth}[3]{%
  \expandafter\defInhrefHid \csname #3Hid\endcsname {#1}{#2}%
  \expandafter\defInhref \csname #3Vis\endcsname {#1}{#2}%
}

\newcommand{\defHrefBothVis}[3]{%
\defHrefBoth{#1}{#2}{#3}%
\expandafter\defHref \csname #3\endcsname {#1}{#2}%
}
\newcommand{\defInhrefBothVis}[3]{%
  \defInhrefBoth{#1}{#2}{#3}%
  \expandafter\defInhref \csname #3\endcsname {#1}{#2}%
}

\newcommand{\defHrefBothHid}[3]{%
\defHrefBoth{#1}{#2}{#3}%
\expandafter\defHrefHid \csname #3\endcsname {#1}{#2}%
}
\newcommand{\defInhrefBothHid}[3]{%
  \defInhrefBoth{#1}{#2}{#3}%
  \expandafter\defInhrefHid \csname #3\endcsname {#1}{#2}%
}

\baselineskip 18pt \textwidth 16cm  \theoremstyle{plain}

\newtheorem*{theorem*}{Theorem}
\newtheorem*{remark*}{Remark}
\newtheorem*{conjecture*}{Conjecture}

\newtheorem{introtheorem}{Theorem}

\newcommand{\C}{\mathbb C}

\renewcommand{\dim}{{\operatorname{dim}}}

\renewcommand{\Hom}{{\operatorname{Hom}}}

\newcommand{\irr}{\operatorname{Irr}}

\usepackage{varioref}

\title{Relative Frobenius Formula}
\author{Avraham Aizenbud, Nir Avni and Yoav Krauz}
\begin{document}

\maketitle
\begin{abstract}
{For a finite group $G$, Frobenius found a formula for the values of the function $\sum_{\Irr G} (\dim\, \pi)^{-s}$ for even integers $s$, where $\Irr G$ is the set of irreducible representations of $G$. We generalize this formula to the relative case: for a subgroup $H$, we find a formula for the values of the function $\sum_{\Irr G} (\dim\, \pi)^{-s} (\dim\, \pi ^H)^{-t}$. We apply our results to compute the E-polynomials of Fock--Goncharov spaces and to relate the Gelfand property to the geometry of generalized Fock--Goncharov spaces.}
\end{abstract}

\tableofcontents

\section{Frobenius' formula}

Let $S$ be a compact surface and let $G$ be a finite group. A {fundamental} formula of Frobenius relates the number of homomorphisms from the fundamental group of $S$ to $G$ and the dimensions of the {irreducible representations} of $G$:

\begin{theorem}\label{thm:frob} Let $S$ be a compact surface of genus $k$ and let $G$ be a finite group. Then,
\[
|G|^{2k-1}\sum_{\pi \in \Irr\, G} (\dim\, \pi )^{2-2k}=|\Hom(\pi_1(S),G)|=| \left\{ (x_1,y_1,\ldots,x_k,y_k)\in G^{2k} \mid [x_1,y_1]\cdots[x_k,y_k]=1 \right\} |,
\]
where $\Irr G$ is the set of (isomorphism classes of) irreducible representations of $G$.
\end{theorem}

For example, $k=0$ gives $\sum_{\pi \in \Irr\, G}(\dim\, \pi)^2=|G|$, whereas from $k=1$ we get
\[
|\Irr\, G|=\frac{1}{|G|}\cdot| \left\{ (x,y)\in G^2 \mid xy=yx \right\} |=\sum_{x\in G} \frac{|C_G(x)|}{|G|}=\sum_{x\in G} \frac{1}{|x^G|}=|G//G|.
\]
{Theorem \ref{thm:frob} also has versions for compact Lie groups and for pro-finite groups (see \cite{Wi,AA}).}

Theorem \ref{thm:frob} is the case $g=1$ of the following theorem:

\begin{theorem}\label{thm:frob.g} Let $G$ be a finite group and let $g \in G$. Then,
\[|G|^{2k-1}\sum_{\pi \in \Irr\, G} (\dim\, \pi )^{1-2k}\chi_\pi(g)=
| \left\{ (x_1,y_1,\ldots,x_k,y_k)\in G^{2k} \mid [x_1,y_1]\cdots[x_k,y_k]=g \right\}|.
\]
\end{theorem}

In this paper, we generalize Frobenius' formula to the relative case, i.e., we replace the representation theory of a group $G$ by the harmonic analysis on some $G$-space $X$. We apply our result for Gelfand pairs and the Hodge theory of Fock--Goncharov spaces.

\section{Relative representation theory}
Relative representation theory is motivated by the following example:
{
\begin{example}
Let $H$ be a (finite) group, and consider $H$ as a $H \times H$-set via the action $$(h_1,h_2) \cdot h:= h_1 h h_2^{-1}.$$ Consider the space $\C[H]$ of complex-valued functions on $H$ as a  representation of $H \times H$. We have $$\C[H]=\bigoplus_{\pi \in \Irr H} \pi \otimes \pi^*.$$
\end{example}
This example shows that understanding the $H \times H$-representation $\C[H]$ ``is the same'' as understanding the representation theory of $H$. One can reformulate many concepts of the representation theory of $H$ in terms of the $H \times H$-representation $\mathbb{C}[H]$. Relative representation theory (also known as abstract harmonic analysis) deals with those concepts considered in a wider generality: a group $G$ acting on a set $X$ and the representation of $G$ on $\C[X]$.}


Two important examples of representation theoretical concepts that have relative counterparts are Schur's Lemma, whose relative counterpart is the Gelfand property (see Definition \ref{def:Gelfand} below) and the notion of a character, whose relative  counterpart is the notion of spherical (or relative) character (see Definition \ref{def:sph.char} below).

\section{Relative version of Frobenius' formula}
We prove the following theorem in \S\ref{sec:pf}:
\begin{theorem}\label{thm:main}
Let $G$ be a finite group acting on a finite set $X$, {let $g\in G$, and let $k\in \mathbb{Z}_{\geq 0}$, $m\in \mathbb{Z}_{\geq 1}$.} Then:
\begin{multline*}
\sum_{\pi \in irr G}\frac{\dim(\Hom_G(\pi, \C[X]))^m}{\dim \pi^{m+2k-1}} \chi_\pi(g)= \frac{1}{\#G^{m+2k-1}} \cdot\\\cdot \#\{p_1,\dots p_m \in X,h_1,\dots h_m,a_1,\dots a_k,b_1,\dots b_k \in G|h_i\in G_{p_i},\prod_{i=1}^m h_i \cdot \prod_{i=1}^k [a_i, b_i]=g\}=\\
=\frac{1}{\#G^{m+2k-1}}\sum_{h_2,\dots h_m,a_1,\dots a_k,b_1,\dots b_k \in G} \#X^{g^{-1}\cdot h_2 \cdots h_m  \cdot  [a_1, b_1] \cdots  [a_k, b_k] }\prod_{i=2}^m \#X^{h_i},
\end{multline*}
where
$[a,b]:=aba^{-1}b^{-1}$ is the commutator of $a$ and $b.$
\end{theorem}

In Appendix \ref{sec:sp.char} we reformulate this theorem in terms of spherical characters.

\section{A criterion for  Gelfand pairs}
Recall the definition of Gelfand pairs:
\begin{definition} \label{def:Gelfand}
Let $G$ be a finite group.
\begin{enumerate}
\item Assume that $G$ acts on a finite set $X$. We say that $X$ is multiplicity free if, for any $\pi \in \irr(G)$, we have
$\dim \Hom_G(\pi,\C[X])\leq 1$.
\item Let $H <G$. We say that $(G,H)$ is a Gelfand pair if $G/H$ is a multiplicity free $G$-set.
\end{enumerate}
\end{definition}
Theorem \ref{thm:main} gives us the following criterion for  Gelfand pairs:
\begin{corollary}
Let $H \subset G$ be a pair of groups, and let $X=G/H$.
Then the pair  $(G,H)$ is a Gelfand pair if and only if
$$\sum_{g,h\in G} \#X^{[g,h]}= \sum_{g,h\in G} \#X^g \cdot \#X^h  \cdot  \#X^{gh}.$$
\end{corollary}

In fact, Theorem \ref{thm:main} implies also the following more general statement: 
\begin{corollary}
Let $H \subset G$ be a pair of groups and let $X=G/H$.
For every $k, m\in \Z_{\geq 0}$
denote:
$$f(k,m):= \sum_{h_1,\dots h_m,a_1,\dots a_k,b_1,\dots b_k \in G} \#X^{h_1 \cdots h_m  \cdot  [a_1, b_1] \cdots  [a_k, b_k] }\prod_{i=1}^m \#X^{h_i}.
$$

Then, the following are equivalent:
\begin{itemize}
\item The pair  $(G,H)$ is a Gelfand pair.
\item For every $k, m\in \Z_{\geq 0}$ and $0 <  l\leq k$, we have
$f(k-l,m)=f(k,m+2l).$
\item For some $k, m\in \Z_{\geq 0}$ and $0<l\leq k$, we have
$f(k-l,m)=f(k,m+2l).$
\end{itemize}

\end{corollary}

\section{Fock--Goncharov spaces}

Theorem \ref{thm:main} can also be interpreted as a counting formula for (generalized) Fock--Goncharov spaces, which we proceed to define. The setting for this section is as follows: let $\overline{S}$ be a compact surface, let $p_1,\ldots,p_m \in \overline{S}$, $m \geq 1$, be distinct points, {and denote $S=\overline{S}\smallsetminus \left\{ p_1,\ldots,p_m \right\}$. Such $S$ is called a surface of finite type}. Choose a base point $s\in S$ and, for each $i=1,\ldots,m$, choose a representative $\tau_i\in \pi_1(S,s)$ from the conjugacy class corresponding to a circle around $p_i$.

\begin{definition} Let $G$ be a group acting on a set $X$. An $X$-framed representation $\pi_1(S,s)\to G$ is a tuple $(\rho,x_1,\ldots,x_m)$, where $\rho : \pi_1(S,s) \rightarrow G$ is a homomorphism, and $x_i\in X$ satisfy $\rho(\tau_i)x_i=x_i$. The collection of all $X$-framed representations is denoted by $\widehat{\mathcal{X}}_{S,s,(\tau_i),G,X}$.
\end{definition}

If $s'$ and $\tau_i'$ are different choices of a point and loops, then there is a bijection (depending on a choice of {a path} from $s$ to $s'$) between $\widehat{\mathcal{X}}_{S,s,(\tau_i),G,X}$ and  $\widehat{\mathcal{X}}_{S,s',(\tau'_i),G,X}$. When no confusion arises, we will omit $s$ and $\tau_i$ from the notation.

If $\mathbf{G}$ is group scheme acting on a scheme $\mathbf X$, then the functor sending a scheme $T$ to $\widehat{\mathcal{X}}_{S,\mathbf G(T),\mathbf X(T)}$ {is} representable by a scheme that we denote by  $\widehat{\mathcal{X}}_{S,\mathbf G, \mathbf X}$.

\begin{definition} Let $\mathbf{G}$ be a group scheme acting on a scheme $\mathbf X$. {Then,}  $\mathbf G$ acts on $\widehat{\mathcal{X}}_{S,\mathbf G, \mathbf X}$, and we denote the quotient stack by $\mathcal{X}_{S,\mathbf G, \mathbf X}$. Similarly, if a group $G$ acts on a set $X$, we denote the quotient groupoid  $G \backslash \widehat{\mathcal{X}}_{S,G,X}$ by ${\mathcal{X}}_{S,G,X}$.
\end{definition}
\begin{remark}
$ $
\begin{itemize}
\item If $\mathbf X$ is the flag variety of a {reductive}  group $\mathbf G$, then the stack  $\mathcal{X}_{S,\mathbf G, \mathbf X}$ was defined in \cite{FG}. The authors of \cite{FG} {defined the notion of a framed $\mathbf{G}$ local system and }showed that {$\mathcal{X}_{S,\mathbf G, \mathbf X}$} is the moduli stack of framed $\mathbf G$ local systems on $S$ (see \cite[\S 2]{FG}). The notion of a framed $\mathbf G$ local system extends to general $\mathbf G$ and $\mathbf X${, and the same proof shows that $\mathcal{X}_{S,\mathbf G, \mathbf X}$ is the moduli space of framed $(\mathbf{G},\mathbf{X})$-local systems}.
\item  If $\mathbf{G}$ is connected, then{, by Lang's Theorem,} $\mathcal{X}_{S,\mathbf G , \mathbf X}(\mathbb{F}_p) \cong \mathcal{X}_{S, \mathbf G(\mathbb{F}_p), \mathbf X (\mathbb{F}_p)}.$
\end{itemize}
\end{remark}

In terms of the definitions above, Theorem \ref{thm:main} implies:
\begin{theorem}\label{thm:vol}
Let ${G}$ be a finite  group  acting on a finite set $X$. Then
$$\#\widehat{\mathcal{X}}_{S, G,  X}= {(\#G)^{1-\chi(S)}}\sum_{\pi \in \irr G}\frac{\dim(\Hom_G(\pi, \C[X]))^{\#(\overline S \smallsetminus S)}}{\dim \pi^{{-\chi(S)}}}, $$
and
$$vol({\mathcal{X}}_{S, G,  X}):=\sum_{x \text{ is an isomorphism class of } \mathcal{X}_{S, G,  X}}  \frac{1}{\# Aut(x)}= {(\#G)^{-\chi(S)}}\sum_{\pi \in \irr G}\frac{\dim(\Hom_G(\pi, \C[X]))^{\#(\overline S \smallsetminus S)}}{\dim \pi^{{-\chi(S)}}}. $$
\end{theorem}

\begin{corollary}
Let $G$ be a finite group acting on a finite set $X$. The following are equivalent:
\begin{itemize}
\item $X$ is a multiplicity free $G$-space.
\item For any two non-compact surfaces {of finite type} $S_1,S_2$ such that $\chi(S_1)=\chi(S_2)$, we have $vol({\mathcal{X}}_{S_1,G,X})=vol({\mathcal{X}}_{S_2,G,X})$.
\item {There are} two non homeomorphic non-compact surfaces {of finite type} $S_1,S_2$ such that $\chi(S_1)=\chi(S_2)$ and $vol({\mathcal{X}}_{S_1,G,X})=vol({\mathcal{X}}_{S_2,G,X})$.
\end{itemize}
\end{corollary}

\begin{definition}
We say that a set $T$ of prime powers  is dense if, for any finite Galois extension $E/\mathbb Q$ and for any conjugacy class {$\gamma \subset Gal(E/\mathbb Q)$}, there exists $p^n\in T$ such that $p$ is unramified in $E$ and $\gamma =Fr_p^n$.
\end{definition}
\begin{remark}
$ $
\begin{itemize}
\item The Chebotarev Density Theorem says that the set of all primes is dense.
\item The Grothendieck trace formula implies that if $X_{1},X_{2}$ are two schemes such that $X_1(\mathbb{F}_q)=X_1(\mathbb{F}_q)$ {when $q$ ranges over a dense set of prime powers}, then  $X_1(\mathbb{F}_{p^n})=X_1(\mathbb{F}_{p^n})$ for almost all primes $p$ and for all natural numbers $n$.
\end{itemize}
\end{remark}

The last corollary and \cite{Kat} implies:
\begin{corollary} Let {$\mathbf{G}$ be a group scheme over $\mathbb{Z}$ acting} on a scheme $\mathbf X$. The following are equivalent:
\begin{itemize}
\item There is a dense set $T$ of prime powers such that, for any $q \in T$, the set $\mathbf X(\mathbb F_{q})$ is a multiplicity free $\mathbf G(\mathbb F_{q})$ space.
\item For all but finitely many primes $p$ and for all $n$, the set $\mathbf X(\mathbb F_{p^n})$ is a multiplicity free $\mathbf G(\mathbb F_{p^n})$ space.
\end{itemize}
Moreover, if these conditions hold then, for any two non-compact surfaces $S_1,S_2$ such that $\chi(S_1)=\chi(S_2)$, the varieties $\widehat{\mathcal{X}}_{S_1,\mathbf{G},\mathbf{X}}$ and $\widehat{\mathcal{X}}_{S_2,\mathbf{G},\mathbf{X}}$ have the same E-polynomial\footnote{For the definition of the E-polynomial see e.g. \cite{Kat}}.
\end{corollary}

We will now apply Theorem \ref{thm:vol} for the case of $\mathbf{GL}_n$ acting on its flag variety $\mathbf{Fl}_{n}$. Recall that, if $\lambda=(\lambda_1,\dots,\lambda_m)$ is a partition of $n$ and $\lambda^*$ is the conjugate partition, then $$h_\lambda(i,j)=\lambda_i-j+\lambda^*_j-i+1$$ is the length of the hook in the Young diagram corresponding to $\lambda$ passing through the box $(i,j)$. We prove the following:


\begin{theorem} \label{thm:vol.FG}
$ $
\begin{itemize}
\item
$$
vol\left ( {\mathcal{X}}_{S,\mathbf{GL}_n,\mathbf{Fl}_n}(\mathbb{F}_q) \right )= (n!)^{\#\overline{S}\smallsetminus S} \sum\limits_{\lambda \text{ is a partition of }n} q^{\sum_k (k-1)\lambda_k \chi(S)}\prod_{i,j\, :\, j \leq \lambda_i}\frac{(q^{h_{\lambda}(i,j)}-1)^{-\chi(S)}}{h_{\lambda}(i,j)^{\#\overline{S}\smallsetminus S}}.
$$
\item The $E$ polynomial of $\widehat{\mathcal{X}}_{S,\mathbf{GL}_n,\mathbf{Fl}_n}$ is
$$
(n!)^{\#\overline{S}\smallsetminus S} \prod_{k=1}^n(x^ny^n-x^ky^k)   \sum\limits_{\lambda} (xy)^{\sum_k (k-1)\lambda_k \chi(S)}\prod_{i,j\, :\, j \leq \lambda_i}\frac{((xy)^{h_{\lambda}(i,j)}-1)^{-\chi(S)}}{h_{\lambda}(i,j)^{\#\overline{S}\smallsetminus S}}.
$$

\end{itemize}
\end{theorem}

For the proof, we collect the following facts:

\begin{proposition}[\cite{jam}] For every partition $\lambda$ of $n$, there exists a unique irreducible representation $R_\lambda$ of $\mathbf{GL}_n(\mathbb{F}_q)$ satisfying: \begin{itemize}
\item $R_\lambda$ appears in the permutation representation $\mathbb{C}[\mathbf{GL}_n(\mathbb{F}_q) / \mathbf{P}_\lambda(\mathbb{F}_q)]$, where $\mathbf{P}_\lambda$ is the standard parabolic corresponding to $\lambda$ {(see \cite[Chapter 11]{jam})}.
\item $R_\lambda$ does not appear in the permutation representation $\mathbb{C}[\mathbf{GL}_n(\mathbb{F}_q) / \mathbf{P}_\mu(\mathbb{F}_q)]$, for $\mu < \lambda$ {(see \cite[Chapter 15]{jam})}.
\item {$$\dim R_\lambda=q^{\sum_k (k-1)\lambda_k} \frac{\#\GL_n(\mathbb{F}_q)}{\prod_{i,j: j \leq \lambda_i} (q^{h_\lambda(i,j)}-1)}.$$} {(see \cite[Page 2]{jam})}.
\end{itemize}
\end{proposition}

Let $\mathbf{B}\subset \mathbf{GL}_n$ be the standard Borel. Taking $T_\lambda=R_\lambda ^{\mathbf{B}(\mathbb{F}_q)}$, we get

\begin{corollary} \label{cor:T.lambda} For every partition $\lambda$ of $n$, {we have} \begin{itemize}
\item $T_\lambda$ appears in the representation $\mathbb{C}[\mathbf{GL}_n(\mathbb{F}_q) / \mathbf{P}_\lambda(\mathbb{F}_q)]^{\mathbf{B}(\mathbb{F}_q)}$.
\item $T_\lambda$ does not appear in the representation $\mathbb{C}[\mathbf{GL}_n(\mathbb{F}_q) / \mathbf{P}_\mu(\mathbb{F}_q)]^{\mathbf{B}(\mathbb{F}_q)}$, for $\mu < \lambda$.
\end{itemize}
\end{corollary}

{The following is classical:}

\begin{proposition} \label{prop:pi.lambda} For every partition $\lambda$ of $n$, there exists a unique irreducible representation $\pi_\lambda$ of $S_n$ satisfying: \begin{itemize}
\item $\pi_\lambda$ appears in the permutation representation $\mathbb{C}[S_n / S_\lambda]$, where $S_{(\lambda_1,\ldots,\lambda_m)}=S_{\lambda_1} \times \cdots \times S_{\lambda_m} \subset S_n$.
\item $\pi_\lambda$ does not appear in the permutation representation $\mathbb{C}[S_n / S_\mu]$, for $\mu < \lambda$.
\item {$\dim \pi_\lambda=\frac{n!}{\prod_{i,j : i \leq \lambda_j}h_\lambda(i,j)}$.}
\end{itemize}
\end{proposition}

\begin{proof}[Proof of Theorem \ref{thm:vol.FG}] Since $\dim \Hom(R_\lambda,\mathbb{C}[\mathbf{Fl}_n])=\dim T_\lambda$, it is enough to show that $\dim T_\lambda = \dim \pi_\lambda$, for every $\lambda$. Recall that the Hecke algebra {$H^{S_n}(t)$} corresponding to the Coxeter group $S_n$ is a (polynomial) one parameter family of algebras whose underlying vector space is $\mathbb{C}[S_n]$; we denote the product in $H^{S_n}(t)$ by $*_t$. Recall that the product $*_1$ is the convolution on $\mathbb{C}[S_n]$ and that, if $t$ is a prime power, then the product $*_t$ corresponds to the convolution in $\mathbb{C}[\mathbf{B}(\mathbb{F}_t) \backslash \mathbf{GL}_n(\mathbb{F}_t)/\mathbf{B}(\mathbb{F}_t)]$ under the identification $\mathbb{C}[\mathbf{B}(\mathbb{F}_t) \backslash \mathbf{GL}_n(\mathbb{F}_t)/\mathbf{B}(\mathbb{F}_t)] \cong \mathbb{C}[S_n]$ given by the Bruhat decomposition. Let $M_\lambda(t) \subset H^{S_n}(t)$ be the subspace of $S_\lambda$-(right)-invariant elements {of $\mathbb{C}[S_n]$}. For every prime power $t$, $M_\lambda(t)$ is an ideal, and, hence, the same is true for every $t$. Using the interpolation of the natural inner product, we get that, for $t\in \mathbb{R}_{\geq 1}$, the algebra $H^{S_n}(t)$ is semisimple, and, hence, there is an (analytic) trivialization of $H^{S_n}(t)$ over $\mathbb{R}_{\geq 1}$. Since there are only finitely many isomorphism types of representations of a given dimension, we get that $M_\lambda(t)$ can also be trivialized over $\mathbb{R}_{\geq 1}$. Corollary \ref{cor:T.lambda} and Proposition \ref{prop:pi.lambda} imply that, under the algebra isomorphism $\mathbb{C}[S_n] \rightarrow \mathbb{C}[\mathbf{B}(\mathbb{F}_q) \backslash \mathbf{GL}_n(\mathbb{F}_q)/\mathbf{B}(\mathbb{F}_q)]$, the modules $T_\lambda$ and $\pi_\lambda$ are isomorphic, and hence have the same dimension.
\end{proof}

\section{Proof of Theorem \ref{thm:main}} \label{sec:pf}
The case $k=0$, $m=1$ of theorem \ref{thm:main} is easy:
\begin{lemma}\label{lem:key}
Let $G$ be a finite group acting on a finite set $X$. Then:
\begin{equation}\label{eq:key}
\sum_{\pi \in \irr G}\dim(\Hom_G(\pi, \C[X]))\cdot \chi_\pi(g) =\chi_{\C[X]}(g)= \#X^g.
\end{equation}
 \end{lemma}

%
%

In order to deduce the general case we  need a basic fact about convolution of characters. Recall that for two functions $f,g \in \C[G]$, the convolution is defined by $$\left(f \ast g\right)(h)=\sum_{u\in G}f(u)g(u^{-1}h).$$

\begin{lemma}\label{lem:conv}
For any $\pi,\tau \in \Irr G$ we have:
$$\chi_\pi \ast \chi_\tau=\frac{\delta_{\pi,\tau} \#G}{\dim(\pi)}\chi_\pi.$$
\end{lemma}
Now we ready to prove the main theorem.
 \begin{proof}[Proof of theorem \ref{thm:main}]
 Applying Lemma \ref{lem:conv}, the assertion follows by convolving  \eqref{eq:key} with itself $m$ times and with  the formula in Theorem \ref{thm:frob.g}.
\end{proof}


\section{Acknowledgments}
We thank Inna Entova Aizenbud for a helpful conversation. A.A. was partially supported by ISF grant 687/13 and a Minerva foundation grant. N.A. was partially supported by NSF grant DMS-1303205. A.A. and N.A. were partially supported by BSF grant 2012247. {N.A. thanks the Weizmann Institute for hospitality.}
\appendix
\section{An alternative proof of the Frobenus  formula}
Lemma \ref{lem:key} gives an alternative proof of the  Frobenius formula (Theorem \ref{thm:frob}).

{Let $G$ be a finite group acting} on a finite set $X$. For a representation $\pi$ of $G$, define a function {on $X \times X$ by}
\begin{equation}\label{eq:sphchar} \chi_\pi^X(x,y)=\frac{1}{\#G}\sum_{h\text{ : } hx=y}\chi_\pi(h).\end{equation}

 \begin{lemma}\label{lem:sp.char}
Consider the 2-sided action of $G\times G$ on $G$. Let $\pi$ be a representation of $G$. Then
$$ \chi_{\pi\otimes \pi^*}^G(1,g)=\frac{1}{\#G \dim\pi}\chi_\pi(g).$$
 \end{lemma}
\begin{proof}
$$ \chi_{\pi\otimes \pi^*}^G(1,g)=\frac{1}{\#G} \sum_{h_{1},h_{2} \text{ : } h_1 h_2^{-1}=g}\chi_\pi(h_1) \chi_\pi(h_1^{-1}) =\frac{(\chi_\pi* \chi_\pi)(g)}{\#G} =\frac{1}{\#G \dim\pi}\chi_\pi(g),$$ where the last equality is by Lemma \ref{lem:conv}
\end{proof}
\begin{proof}[Proof of Theorem \ref{thm:frob}]
the case $k=1$ follows from the Lemma \ref{lem:key} and lemma \ref{lem:sp.char}. The general case follows by taking  convolution power of the case $k=1$ and using  Lemma \ref{lem:conv}.
\end{proof}
\section{The spherical character}\label{sec:sp.char}
The relative counterpart of the notion of the  character of a representation is given in the following definition:
\begin{definition} \label{def:sph.char}
Let $G$ be a finite group {acting} on a finite set $X$. Let $\pi$ be a representation of $G$.
\begin{enumerate}
\item Let $\phi:\pi \to \C[X]$ and $\psi:\pi^* \to \C[X]$ be morphisms of representations. 
Denote by $\phi^t$ and $\psi^t$ the dual maps.
We define the spherical character {$\chi_\pi^{\phi\otimes\psi}\in \C[X \times X]$} by $$\chi_\pi^{\phi\otimes \psi}(x,y)=\langle\phi^t(\delta_x),\psi^t(\delta_y) \rangle,$$
 where $\delta_x\in \C[X]=\C[X]^*$ is the Kronecker  delta function supported at $x$.
\item This definition extends (by linearity) to the case when $\phi\otimes \psi$ is replaced by any element of $Hom(\pi,\C[X]) \otimes  Hom(\pi^*,\C[X])=End(Hom(\pi,\C[X]))$.
\end{enumerate}
\end{definition}
 \begin{lemma}\label{lem:sp.char.int}
$$\chi_\pi^X:=\chi_\pi^{Id_{Hom(\pi,\C[X])}}.$$
 \end{lemma}
\begin{proof}
For $x\in X$, let $L_\pi^x:\Hom_G (\pi,\C[X]) \to \pi^*$ be the linear map defined by
$$\phi\in\Hom_G (\pi,\C[X])\mapsto \left(u\in\pi\mapsto \phi(u)(x)\right).$$
Note that $\Hom_G (\pi,\C[X])$, $\Hom_G(\pi^*,\C[X])$ are naturally dual to each other by the pairing
$$\langle \phi,\psi \rangle := \sum_{x\in X} \langle L_\pi^x \phi, L_{\pi^*}^x \psi \rangle
\qquad \left(\phi\in\Hom_G (\pi,\C[X]);\ \psi\in\Hom_G(\pi^*,\C[X])\right)$$
therefore we shall identify $\Hom_G (\pi^*,\C[X])$ with $\Hom_G(\pi,\C[X])^*$.

Let $\phi\in\Hom_G (\pi,\C[X])$, $\psi\in\Hom_G(\pi^*,\C[X])$. Then by definition,
$$\chi_\pi^{\phi\otimes \psi}(x,y)=\langle\phi^t(\delta_x),\psi^t(\delta_y) \rangle=\langle L_\pi^x \phi,L_{\pi^*}^y \psi \rangle=\langle\left(L_{\pi^*}^y\right)^t L_\pi^x \phi, \psi \rangle,$$
{so $\chi_\pi^{Id_{Hom(\pi,\C[X])}}(x,y)=\tr\left(\left(L_{\pi^*}^y\right)^t L_\pi^x\right)$.}

It is easy to see that $\left(L_{\pi}^x\right)^t:\pi\to\Hom_G(\pi^*,\C[X])$ can be computed by
$$\forall u\in\pi,f\in\pi^*:
\left(\left(L_{\pi}^x\right)^t u\right)(f)=\frac{1}{\#G}\sum_{h\in G} f\left(\pi(h)u\right)\delta_{hx}.
$$

Now, {$\chi_\pi^{Id_{Hom(\pi,\C[X])}}=\tr\left(\left(L_{\pi^*}^y\right)^t L_\pi^x\right)=\tr\left(L_\pi^x\left(L_{\pi^*}^y\right)^t \right)$.} Note that $L_\pi^x\left(L_{\pi^*}^y\right)^t$ is the linear mapping $\pi^*\to\pi^*$ defined by
$$
\forall f\in\pi^*: \left(L_\pi^x\left(L_{\pi^*}^y\right)^t\right)f=\left(u\in\pi\mapsto
\frac{1}{\#G}\sum_{h\in G} \langle u,\left(\pi^*(h)f\right)\rangle\delta_{hy,x}\right)=\frac{1}{\#G}\sum_{h\text{ s.t. }hy=x}\pi^*(h)f
$$
{so $$\chi_\pi^{Id_{Hom(\pi,\C[X])}}=\tr\left(L_\pi^x\left(L_{\pi^*}^y\right)^t \right)=\frac{1}{\#G}\sum_{h\text{ s.t. }hy=x}\chi_{\pi^*}(h)=\frac{1}{\#G}\sum_{h\text{ s.t. }hx=y}\chi_{\pi}(h)=\chi_\pi ^X(x,y)$$}

\end{proof}

We reformulate Theorem \ref{thm:main} in terms of the spherical character:

\begin{theorem}\label{thm:main.sph}
Let $G$ be a finite group that acts on a  finite set $X$. Then:
\begin{multline*}
\sum_{\pi \in \irr G}\frac{\dim(\Hom_G(\pi, \C[X]))^m}{\dim \pi^{m+2k-1}} \chi^{X}_\pi(x_{1},x_{2})= \frac{1}{\#G^{m+2k}} \cdot\\\cdot \#\{p_1,\dots p_m \in X,h_1,\dots h_m,a_1,\dots a_k,b_1,\dots b_k \in G|h_i\in G_{p_i},\prod_{i=1}^m h_i \cdot \prod_{i=1}^k [a_i, b_i]\cdot x_1=x_2\}.
\end{multline*}
\end{theorem}

\end{document}